\documentclass[a4paper,11pt]{article}
\usepackage[latin1]{inputenc}
\usepackage[T1]{fontenc}

\usepackage{array}
\usepackage{fancyhdr}
\usepackage{longtable}
\usepackage{tabularx}
\usepackage{rotating}

\usepackage{makeidx}
\usepackage{amsmath}
\usepackage{amsthm}
\usepackage{amssymb}
\usepackage[mathscr]{eucal}
\usepackage{enumerate}
\usepackage[dvips]{epsfig}
\usepackage{float}  
\usepackage{ifthen} 
\usepackage{extarrows}  
\usepackage{stmaryrd}   
\usepackage{arydshln}     
\usepackage{multirow}
\usepackage{mathtools}
\usepackage[dvips]{geometry} 
\usepackage{listings}
\usepackage{subfig}
\usepackage{dsfont}
\usepackage{xspace}

\usepackage{hyperref}

\theoremstyle{plain}
\newtheorem{thm}{Theorem}[section]
\newtheorem{prop}[thm]{Proposition}

\theoremstyle{definition}

\theoremstyle{remark}

\title{{An Analysis of the Milstein Scheme for SPDEs without a Commutative Noise Condition}}

\author{Claudine von Hallern$^1$\thanks{This project was funded by the Cluster of Excellence "The Future Ocean". 
The "Future Ocean" is funded within the framework of the Excellence Initiative 
by the Deutsche Forschungsgemeinschaft
(DFG) on behalf of the German federal and state governments.
e-mail: vonhallern@math.uni-kiel.de, leonhard@math.uni-kiel.de}
\ \ and Andreas R\"o\ss ler$^2$\thanks{e-mail: roessler@math.uni-luebeck.de}
\bigskip
\\
\small{$^1$Department of Mathematics, Christian-Albrechts-Universit\"at zu Kiel,} \\
\small{Christian-Albrechts-Platz 4, 24118 Kiel, Germany} \\[0.2cm]
\small{$^2$Institute of Mathematics, Universit\"at zu L\"ubeck,} \\
\small{Ratzeburger Allee 160, 23562 L\"ubeck, Germany} 
}

\date{}

\begin{document}

\maketitle

\begin{abstract}
In order to approximate solutions of stochastic partial differential equations (SPDEs) 
that do not possess commutative noise, one has to simulate the involved 
iterated stochastic integrals. Recently, two approximation methods for iterated 
stochastic integrals in infinite dimensions were 
introduced in~\cite{MR3949104}. As a result of this, it is now possible to apply 
the Milstein scheme by Jentzen and R\"ockner~\cite{MR3320928} to equations 
that need not fulfill the commutativity condition. We prove that the order 
of convergence of the 
Milstein scheme can be maintained when combined with one of the two approximation 
methods for iterated stochastic integrals. However, we also have to 
consider the computational cost and the corresponding effective order of 
convergence for a meaningful comparison with other schemes. An analysis 
of the computational cost shows that, in dependence on the equation, a 
combination of the Milstein scheme with both of the two methods may be 
the preferred choice. Further, the Milstein scheme is compared to the exponential 
Euler scheme and we show for different SPDEs depending on the parameters describing, 
e.g., the regularity of the equation, which one of the schemes achieves the 
highest effective order of convergence.
\end{abstract}
%

\section{Introduction}\label{HR:Sec:Intro}
It is well known that for a commutative stochastic differential equation 
the Milstein scheme can be easily implemented as no iterated stochastic 
integrals have to be simulated. However, if we deal with an SPDE which 
does not fulfill the commutativity condition, it is, in general, not 
possible to rewrite the expression in such a way that implementation 
becomes straightforward. 
In the following, we consider SPDEs of type
\begin{equation}\label{HR:SPDE}
  \mathrm{d} X_t = \big( AX_t+F(X_t)\big) \, \mathrm{d}t + B(X_t) \, \mathrm{d}W_t, 
  \quad t\in(0,T], \quad X_0 = \xi.
\end{equation}
In this work, we are  concerned about 
the efficient approximation of the mild solution of
equations \eqref{HR:SPDE} which do not need to have commutative noise
by a higher order scheme, that is, we deal with equations where the 
commutativity condition 
\begin{equation}\label{HR:Comm}
    \big( B'(v) (B(v) u) \big) \tilde{u} =
    \big( B'(v) (B(v) \tilde{u}) \big) u
\end{equation}
for all $v\in H_{\beta}$, $u, \tilde{u} \in U_0$ does \emph{not} have to be fulfilled.
We consider the Milstein scheme for SPDEs 
recently proposed in~\cite{MR3320928} which reads as
$Y_0^{N,K,M} = P_N\xi$ and
\begin{equation}\label{HR:Milstein}
  \begin{split}
  Y_{m+1}^{N,K,M} &= P_N e^{Ah} \bigg(Y_m^{N,K,M} + hF(Y_m^{N,K,M}) 
  + B(Y_m^{N,K,M}) \Delta W^{K,M}_m \\
  &\quad + \int_{t_m}^{t_{m+1}} B'(Y_m^{N,K,M}) \Big(\int_{t_m}^{s}
  B(Y_m^{N,K,M}) \, \mathrm{d}W^K_r\Big)
  \, \mathrm{d}W^K_s \bigg)
  \end{split}
\end{equation}
for some $N,M,K \in\mathbb{N}$, $h=\frac{T}{M}$ and $m\in\{0,\ldots,M-1\}$. 
For details on the notation, we refer to Section~\ref{HR:Sec:Setting} and 
Section~\ref{HR:Sec:Scheme}.
The main difficulty in the approximation of equations with non-commutative 
noise is the simulation of the
iterated stochastic integrals, since it is not possible to rewrite 
integrals such as
\begin{equation*}
  \int_t^{t+h}B'(X_t)\Big(\int_t^sB(X_t)\,\mathrm{d}W_r^K\Big)\,\mathrm{d}W_s^K
\end{equation*}
for $h>0$, $t,t+h\in[0,T]$ and $K\in\mathbb{N}$ 
in terms of increments of the approximation $(W_t^K)_{t\in[0,T]}$ 
of the $Q$-Wiener process $(W_t)_{t\in[0,T]}$
like in the commutative case, see~\cite{MR3320928}.
Since the iterated stochastic integrals can, in general,
not be computed explicitly, we need to approximate these terms.
In~\cite{MR3949104}, the authors recently proposed two algorithms
to approximate integrals of type
\begin{equation}\label{HR:IteratedIntSPDE}
 \int_t^{t+h} \Psi\left(\Phi \int_t^s  \mathrm{d}W_r\right) \mathrm{d}W_s
\end{equation}
with $t\geq 0$, $h>0$
for some operators $\Psi\in L(H,L(U,H)_{U_0})$, $\Phi\in L(U,H)_{U_0}$ 
and a $Q$-Wiener process $(W_t)_{t\in[0,T]}$. 
Applying these algorithms, it is possible to implement the Milstein scheme stated in
\eqref{HR:Milstein} if we choose $\Psi = B'(Y_t)$ and $\Phi = B(Y_t)$ for some 
$B : H_{\beta} \rightarrow L(U, H)_{U_0}$
and an approximation $Y_t \in H_{\beta}$ with $t \geq 0$ and 
$\beta\in [0, 1)$. For more details on the operators, we refer 
to~\cite{MR3949104} and Section~\ref{HR:Sec:Setting}.  
In this work, we combine the Milstein scheme with the approximation of 
the iterated stochastic integrals. 
\\ \\
For finite dimensional stochastic differential equations, the issue of 
how to simulate iterated stochastic integrals was answered, e.g., by 
\cite{MR1178485} and \cite{MR1843055}. In this setting, the Milstein 
scheme combined with the approximation as specified by \cite{MR1843055} 
outperforms the method that was introduced in \cite{MR1178485}
in terms of the computational cost when the temporal order of convergence 
of the Milstein scheme is to be preserved. The results in \cite{MR3949104} 
suggest that in the infinite dimensional setting of SPDEs, it is not obvious 
which of the two methods requires less computational effort. Therefore, in 
this work, we analyze the cost involved in the simulation for each of the 
two methods in detail and also compare the Milstein scheme combined 
with each method to the exponential Euler scheme.
\section{Analysis of the Numerical Scheme}
We present two versions of the Milstein scheme for non-commutative SPDEs 
in this section. To be precise, we analyze two schemes which differ by the 
method that is used to approximate the iterated stochastic integrals that 
are involved. We prove in Section~\ref{HR:Sec:Scheme} that the order of 
convergence that the Milstein scheme obtains for commutative equations, 
see \cite{MR3320928}, can be maintained if the iterated integrals are 
approximated by the methods introduced in \cite{MR3949104}. In 
Section~\ref{HR:Sec:Cost}, these two versions of the Milstein scheme are 
compared to each other and to the exponential Euler scheme when the 
computational cost is also taken into account.
\subsection{Setting and Assumptions}\label{HR:Sec:Setting}
The setting that we work in is similar to the one considered for the 
Milstein scheme in \cite{MR3320928} except that the commutativity 
condition (24) in their paper (see also equation~\eqref{HR:Comm}) does 
not have to be fulfilled, that we replace the space $L_{\text{HS}}(U_0,H)$ 
by $L(U,H)_{U_0}\subset L_{\text{HS}}(U_0,H)$ in assumption (A3) and that 
we introduce a projection operator in (A3).
\\ \\
Let $T\in(0,\infty)$ be fixed, let $(H,\langle \cdot,\cdot \rangle_H)$
and $(U,\langle \cdot,\cdot\rangle_U)$ denote some separable real-valued
Hilbert spaces.
We fix some probability space $(\Omega,\mathcal{F},P)$ and denote by 
$(W_t)_{t\in[0,T]}$ a $U$-valued 
$Q$-Wiener process with respect to the
filtration $(\mathcal{F}_t)_{t\in [0,T]}$ 
which fulfills the usual conditions. The operator 
$Q \in L(U)$ is assumed to be nonnegative, symmetric and to
have finite trace. We denote its eigenvalues 
by $\eta_j$ with corresponding eigenvectors $\tilde{e}_j$ for $j\in\mathcal{J}$
with some countable index set $\mathcal{J}$ forming an orthonormal basis
of $U$~\cite{MR2329435}. 
We employ the following series representation
of the $Q$-Wiener process, see~\cite{MR2329435},
\begin{equation*}
   W_t = \sum_{\substack{j\in\mathcal{J} \\ \eta_j \neq 0}} 
   \sqrt{\eta_j} \, \tilde{e}_j \, \beta_t^j, \quad t\in[0,T].
\end{equation*}
Here, $(\beta_t^j)_{t\in[0,T]}$ denote independent 
real-valued Brownian motions for all $j\in\mathcal{J}$
with $\eta_j\neq 0$.
By means of the operator $Q$, we define the subspace
$U_0 \subset U$ as $U_0 =  Q^{\frac{1}{2}}U$.
The set 
of Hilbert-Schmidt operators mapping from $U$ to $H$ is denoted by $L_{\text{HS}}(U,H)$
and the space 
of linear bounded operators on $U$ restricted to $U_0$ by
$(L(U,H)_{U_0},\|\cdot\|_{L(U,H)})$  with
$L(U,H)_{U_0} := \{T \colon U_0 \to H \, | \, T\in L(U,H) \}$.
Moreover, we designate $L^{(2)}(U,H) = L(U,L(U,H))$ 
and $L_{\text{HS}}^{(2)}(U,H)= L_{\text{HS}}(U,L_{\text{HS}}(U,H))$.\\ \\
Our aim is to approximate the mild solution of SPDE~\eqref{HR:SPDE} 
and, therefore, we impose the following assumptions.
\begin{description}
  \item[(A1)]  The linear operator 
  $A \colon \mathcal{D}(A)\subset H \to H$ generates an analytic semigroup
  $S(t) = e^{At}$ for all $t\geq 0$. 
  Let $\lambda_i \in (0,\infty)$ denote the eigenvalues of $-A$ with 
  eigenvectors $e_i$ for $i \in \mathcal{I}$ and some countable index set $\mathcal{I}$,
  i.e., it holds $-Ae_i =\lambda_i e_i$ for all $i\in\mathcal{I}$. Moreover, assume that
  $\inf_{i\in\mathcal{I}}\lambda_i >0$ and that
  the eigenfunctions $\{ e_i : i\in\mathcal{I}\}$ of $-A$
  form an orthonormal basis of $H$, see~\cite{MR1873467}.
  Furthermore, 
  \begin{equation*}
    Av = \sum_{i\in\mathcal{I}} -\lambda_i\langle v,e_i\rangle_H e_i
  \end{equation*}
  for all $v\in \mathcal{D}(A)$. By means of $A$, we define the real Hilbert
  spaces
  $H_r := \mathcal{D}((-A)^r)$ for $r\in[0,\infty)$
  with norm $\|x\|_{H_r} =\|(-A)^rx\|_H$ for
  $x\in H_r$.
  \item[(A2)] For some $\beta\in[0,1)$, assume that 
  $F \colon H_{\beta} \to H$ is twice continuously Fr\'{e}chet differentiable
  with $\sup_{v\in H_{\beta}} \|F'(v)\|_{L(H)} < \infty$ and 
  $\sup_{v\in H_{\beta}} \|F''(v)\|_{L^{2}(H_{\beta},H)}<\infty$.
  \item[(A3)]
  The operator $B \colon H_{\beta} \to L(U,H)_{U_0}$ 
  is twice continuously Fr\'{e}chet 
  differentiable with 
  $\sup_{v\in H_{\beta}} \|B'(v)\|_{L(H,L(U,H))} < \infty$,
  $\sup_{v\in H_{\beta}} \|B''(v)\|_{L^{(2)}(H_{\beta},L_{\text{HS}}(U_0,H)))}<\infty$. 
  Assume that $B(H_{\delta}) \subset L_{\text{HS}}(U_0,H_{\delta})$ for some $\delta
  \in (0,\tfrac{1}{2})$ and that 
  \begin{align*}
      \| B(u) \|_{L_{\text{HS}}(U_0,H_{\delta})} &\leq C ( 1 +
      \| u \|_{H_{\delta}} ) , \\
      \| B'(v) PB(v) - B'(w) PB(w) \|_{L_{\text{HS}}^{(2)}(U_0,H)} &\leq C
      \| v - w \|_{H} , \\
      \| (-A)^{-\vartheta} B(v) Q^{-\alpha} \|_{L_{\text{HS}}(U_0,H)} &\leq C
      (1 + \| v \|_{H_{\gamma}})
  \end{align*}
  for some constant $C>0$, all $u \in H_{\delta}$, $v, w \in H_{\gamma}$, where
  $\gamma \in \left[ \max( \beta, \delta),
  \delta + \frac{1}{2} \right)$, $\alpha \in (0,\infty)$, $\vartheta \in \left( 0,
  \frac{1}{2} \right)$, 
  any projection operator $P \colon H \to \text{span}\{e_i: i\in\tilde{\mathcal{I}}\}\subset H$
  with finite index set $\tilde{\mathcal{I}}\subset \mathcal{I}$ and the 
  case that $P$ is the identity.
  \item[(A4)] Assume that the initial value $\xi \colon \Omega \to H_{\gamma}$
  fulfills $\mathrm{E}\big[\|\xi\|^4_{H_{\gamma}}\big] <\infty$ and that it is
  $\mathcal{F}_0$-$\mathcal{B}(H_{\gamma})$-measurable.
\end{description}
  
In the following, we do not distinguish between the operator $B$ and its extension 
$\tilde{B} \colon H \to L(U,H)_{U_0}$ which is globally Lipschitz continuous; this holds as 
$H_{\beta}\subset H$ is dense. With $F$, we proceed analogously. 
Conditions (A1)--(A4) imply the existence of a unique mild solution 
$X \colon [0,T] \times \Omega \to H_{\gamma}$ for SPDE
\eqref{HR:SPDE}, see~\cite{MR2852200,MR3320928}.
\subsection{The Milstein Scheme for Non-commutative SPDEs}
\label{HR:Sec:Scheme}
We define the numerical scheme under consideration
and introduce the corresponding discretizations of the infinite dimensional spaces. 
To be precise, we need to discretize the time 
interval $[0,T]$, project the Hilbert space $H$ 
to some finite dimensional subspace and we need an
approximation of the infinite dimensional stochastic process 
$(W_t)_{t\in[0,T]}$. 
For the discretization of the solution
space $H$, we define a projection operator
$P_N \colon H \to H_N$ that maps $H$ to the finite dimensional subspace 
$H_N := \text{span}\{e_i : i\in\mathcal{I}_N\} \subset H$
for some fixed $N \in \mathbb{N}$. This projection is expressed by the index
set $\mathcal{I}_N \subset \mathcal{I}$ 
with $|\mathcal{I}_N|=N$ that picks $N$
basis functions. We specify this operator as
\begin{equation*}
  P_N x = \sum_{i\in\mathcal{I}_N} \langle x,e_i\rangle_H e_i, \quad x\in H.
\end{equation*}
Similarly, we approximate the $Q$-Wiener process. 
For $K\in\mathbb{N}$, we define the projected $Q$-Wiener process 
$(W^K_t)_{t\in[0,T]}$ taking values in
$U_K := \text{span}\{\tilde{e}_j : j\in\mathcal{J}_K\} \subset U$
by
\begin{equation*}
  W_t^K := 
  \sum_{j\in\mathcal{J}_K} \sqrt{\eta_j} \tilde{e}_j \beta_t^j,
  \quad t\in[0,T],
\end{equation*}
for some index set $\mathcal{J}_K \subset \mathcal{J}$ with $|\mathcal{J}_K| = K$ and 
$\eta_j\neq 0$ for $j\in\mathcal{J}_K$.
For the temporal discretization, we choose 
an equidistant time step for legibility of the representation.
Let $h = \frac{T}{M}$ for some $M\in\mathbb{N}$ and denote $t_m = m\cdot h$ 
for $m\in\{0,\ldots,M\}$.
On this grid, we define the increments of the projected $Q$-Wiener process
\begin{equation*}
    \Delta W_m^{K,M} := W_{t_{m+1}}^K - W_{t_m}^K
    = \sum_{j \in \mathcal{J}_K}
    \sqrt{\eta_j} \, \Delta \beta_m^j \, \tilde{e}_j 
\end{equation*}
where the increments of the real-valued Brownian
motions are given by  $\Delta \beta_m^j = 
\beta_{t_{m+1}}^j-\beta_{t_m}^j$
for $m \in \{0,\ldots, M-1\}$, $j\in\mathcal{J}_K$. 
We apply these discretizations
to the setting described above. Then, the Milstein 
scheme yields a discrete-time stochastic process which we denote by 
$(\bar{Y}_m^{N,K,M})_{m \in \{0,\ldots, M\}}$ such that $\bar{Y}_m^{N,K,M}$ 
is $\mathcal{F}_{t_m}$-$\mathcal{B}(H)$-measurable
for all $m \in \{0,\ldots,M\}$, $M\in\mathbb{N}$. We define the 
Milstein scheme (MIL) for non-commutative SPDEs
based on~\cite{MR3320928} as
$\bar{Y}_{0}^{N,K,M} = P_N \xi$ and
\begin{equation} \label{HR:MIL}
  \begin{split}
  \bar{Y}^{N,K,M}_{m+1} &= P_Ne^{Ah}\bigg(\bar{Y}^{N,K,M}_m + hF(\bar{Y}^{N,K,M}_m)
  + B(\bar{Y}^{N,K,M}_m)\Delta W_{m}^{K,M} \\
  &\quad + \sum_{i,j\in\mathcal{J}_K} B'(\bar{Y}^{N,K,M}_m)
  \big(P_N B(\bar{Y}^{N,K,M}_m)\tilde{e}_i,\tilde{e}_j\big)  \bar{I}^Q_{(i,j),m}\bigg)
  \end{split}
\end{equation}
for all $m\in\{0,\ldots,M-1\}$.
Compared to the Milstein scheme \eqref{HR:Milstein} proposed in \cite{MR3320928}, we added
an additional projector and replaced the iterated stochastic integrals
\begin{equation*}
 I^Q_{(i,j),m} := 
 \int_{t_m}^{t_{m+1}} \int_{t_m}^s \langle \mathrm{d} W_r, \tilde{e}_i \rangle_U 
 \langle \mathrm{d}W_s, \tilde{e}_j \rangle_U 
\end{equation*}
by an approximation $\bar{I}^Q_{(i,j),m}$
for all $i,j\in\mathcal{J}_K$ and $m\in\{0,\ldots,M-1\}$.
We can show that the error estimate for the Milstein approximation
that is obtained in 
the commutative case remains valid for the scheme MIL in
\eqref{HR:MIL} if $\bar{I}^Q_{(i,j),m}$ 
represents an approximation obtained by one of the methods 
introduced in \cite{MR3949104} provided the accuracy 
for these approximations is chosen appropriately. 
If Algorithm~1 in \cite{MR3949104} is employed to 
approximate the iterated integrals, we denote the numerical 
scheme~\eqref{HR:MIL} by 
MIL1 and the approximation $\bar{I}^{Q}_{(i,j),m}$ of 
$I^Q_{(i,j),m}$ is denoted by 
$\bar{I}^{Q,(D),(1)}_{(i,j),m}$.
This algorithm is based on a series representation of the 
iterated stochastic integral which is truncated after $D$ summands for some 
$D\in\mathbb{N}$, see~\cite{MR1178485,MR3949104}. 
If we employ Algorithm~2 instead, the scheme~\eqref{HR:MIL} is called MIL2 and 
we denote the approximation $\bar{I}^{Q}_{(i,j),m}$ of $I^Q_{(i,j),m}$ by
$\bar{I}^{Q,(D),(2)}_{(i,j),m}$. The main difference compared 
to Algorithm~1 is that the 
series is not only truncated but the remainder is
approximated by a multivariate normally distributed random
vector additionally, see~\cite{MR3949104,MR1843055} for details.
Let
\begin{align*}
  \mathcal{E}(M,K) &= \bigg( \mathrm{E}\bigg[\Big\|\int_{t_l}^{t_{l+1}}B'(\bar{Y}_l)
  \Big(\int_{t_l}^s P_N B(\bar{Y}_l) \, \mathrm{d}W_r^K\Big) \, \mathrm{d}W^K_s \\
  &\quad  - \sum_{i,j\in\mathcal{J}_K} \bar{I}_{(i,j),l}^Q B'(\bar{Y}_l)
  (P_N B(\bar{Y}_l)\tilde{e}_i, \tilde{e}_j)\Big\|_H^2\bigg] \bigg)^{\frac{1}{2}}
\end{align*}
for all $l\in\{0,\ldots,m-1\}$, $m\in\{1,\ldots,M\}$
and $M,K\in\mathbb{N}$ denote the approximation error of the 
iterated integral term. 
Then, we obtain the following error estimate.
\begin{thm}[Convergence of Milstein scheme] \label{HR:Thm:Milstein}
   Let assumptions (A1)--(A4) hold. Then, there exists a
    constant $C_{Q,T} \in (0,\infty)$, independent of $N$, $K$ and $M$, such
    that for $(\bar{Y}_m^{N,K,M})_{0 \leq m \leq M}$, defined by the Milstein scheme in
    \eqref{HR:MIL}, it holds
    \begin{align*}
     & \Big( \mathrm{E}\Big[ \big\| X_{t_m} - \bar{Y}_m^{N,K,M} \big\|_H^2
      \Big] \Big)^{\frac{1}{2}} \\
      & \quad\leq C_{Q,T} \Big( \Big( \inf_{i \in \mathcal{I} \setminus
      \mathcal{I}_N} \lambda_i \Big)^{-\gamma}
      + \Big( \sup_{j \in \mathcal{J} \setminus \mathcal{J}_K}
      \eta_j \Big)^{\alpha}
     + M^{-q_{\text{MIL}}}
      +\mathcal{E}(M,K)M^{\frac{1}{2}}\Big)
    \end{align*}
    with $q_{\text{MIL}} = \min(2(\gamma-\beta),\gamma)$ and 
    for all $m \in \{0,\ldots, M\}$ and all $N,K,M \in
    \mathbb{N}$. The parameters are determined by assumptions (A1)--(A4). 
\end{thm}
The proof of this statement is given at the end of this section. 
\\ \\
Depending on the choice of the algorithm, we get a different error 
bound for $\mathcal{E}(M,K)$. We set $\Psi = B'(\bar{Y}_l)$ 
and $\Phi = P_NB(\bar{Y}_l)$ in~\eqref{HR:IteratedIntSPDE}. Then, we 
can transfer the error estimates given in~\cite[Corollary~1, 
Corollary~2, Theorem~4]{MR3949104} to our setting. Thus,
for Algorithm~1 there exists some constant $C_{Q,T}>0$ such that 
\begin{align} \label{HR:Err-Alg1}
  \mathcal{E}(M,K) = \mathcal{E}^{(D),(1)}(M,K) \leq C_{Q,T} \frac{1}{M \, \sqrt{D}}
\end{align}
for all $D,K,M\in\mathbb{N}$.
In contrast, for Algorithm~2, we get an estimate 
that converges in $D$ with order~$1$. It is,
however, also dependent on the number $K$ which
controls the approximation of the $Q$-Wiener process
as well as on the eigenvalues 
$\eta_j$, $j \in \mathcal{J}_K$,
of the operator $Q$.
There exists some constant $C_{Q,T}>0$ such that 
\begin{align} \label{HR:Err-Alg2}
  \mathcal{E}(M,K) = \mathcal{E}^{(D),(2)}(M,K) \leq C_{Q,T} 
  \frac{\min \big(  K \sqrt{K-1}, ( \min_{j\in\mathcal{J}_K} \eta_j )^{-1} \big)}{M \, D}
\end{align}
for all 
$D, K, M \in\mathbb{N}$. For example, if we assume 
$\eta_j \asymp j^{-\rho_Q}$,
$\rho_Q > 1$ and all 
$j \in \mathcal{J}=\mathbb{N}$, then 
in the case $\rho_Q<\frac{3}{2}$ it holds $\mathcal{E}^{(D),(2)}(M,K) \leq C_{Q,T} 
(\min_{j\in\mathcal{J}_K} \eta_j)^{-1} \, M^{-1} \, D^{-1}$
and $\mathcal{E}^{(D),(2)}(M,K) \leq C_{Q,T}  \,
K(K-1)^{\frac{1}{2}} \, M^{-1} \, D^{-1}$ for $\rho_Q \geq \frac{3}{2}$.
The proofs of these error estimates can be found in~\cite{MR3949104}.
It is not immediately obvious which of the two algorithms 
is superior, see 
also~\cite{MR3949104} for a discussion of this issue. 
Here, we repeat the considerations in short.
For the two algorithms stated above, we want to select the
integer $D$ such that the order of convergence stated
in Theorem~\ref{HR:Thm:Milstein} is not reduced. 
Therefore, we need to choose $D \geq M^{2\min(2(\gamma-\beta),\gamma)-1}$
for Algorithm~1. 
In contrast, for Algorithm~2, we require  
$D \geq M^{\min(2(\gamma-\beta),\gamma)-\frac{1}{2}} (\min_{j\in\mathcal{J}_K}\eta_j)^{-1}$
or $D \geq M^{\min(2(\gamma-\beta),\gamma)-\frac{1}{2}} K \sqrt{K-1}$. 
Alternatively, one can choose
$D \geq M^{-1} (\sup_{j \in \mathcal{J} \setminus \mathcal{J}_K}\eta_j )^{-2\alpha}$
for Algorithm~1 and 
$\displaystyle D \geq M^{-\frac{1}{2}}(\min_{j\in\mathcal{J}_K}\eta_j)^{-1}
( \sup_{j \in \mathcal{J} \setminus \mathcal{J}_K}\eta_j )^{-\alpha}$
or $\displaystyle D \geq M^{-\frac{1}{2}} K \sqrt{K-1} ( \sup_{j \in \mathcal{J} \setminus \mathcal{J}_K}
\eta_j )^{-\alpha}$ for Algorithm~2.
This shows that the choice of
$D$ depends on $\gamma$, $\beta$,
$K$, $(\eta_j)_{j\in\mathcal{J}_K}$ and on $\alpha$ additionally.
Therefore, the choice of $D$, and with this
the computational effort for the simulation of the iterated stochastic
integrals is dependent on the equation to be solved.
We cannot identify one scheme to be superior in general and refer 
to Section~\ref{HR:Sec:Cost} for details. 
Now, we prove the statement on the convergence of the 
schemes MIL1 and MIL2. 

\begin{proof} [Proof of Theorem~\ref{HR:Thm:Milstein}]
The proof of convergence of the Milstein scheme in \cite{MR3320928} 
does not use the commutativity assumption, therefore, it remains 
valid also in our setting. To ease the notations, we denote by $(Y_m)_{m\in\{0,\ldots,M-1\}}$ 
the Milstein approximation which does not involve an approximation of the 
iterated stochastic integrals
\begin{equation}\label{HR:MILcomm}
  \begin{split}
  Y_{m+1} &= P_N e^{Ah} \bigg(Y_m + h F(Y_m) 
  + B(Y_m) \Delta W^{K,M}_m \\
  &\quad + \int_{t_m}^{t_{m+1}} B'(Y_m) \Big(\int_{t_m}^{s}
  P_N B(Y_m) \, \mathrm{d}W^K_r\Big)
  \, \mathrm{d}W^K_s \bigg).
  \end{split}
\end{equation}
Analogously to the proof for Theorem~1 in~\cite{MR3320928}, we get an estimate for
\eqref{HR:MILcomm} of the form
\begin{align*}
&\Big(\mathrm{E}\big[\|X_{t_m}-Y_m\|_H^2\big]\Big)^{\frac{1}{2}} 
\\ 
& \quad 
\leq
C_{Q,T} \Big( \Big( \inf_{i \in \mathcal{I} \setminus
    \mathcal{I}_N} \lambda_i \Big)^{-\gamma}
    + \Big( \sup_{j \in \mathcal{J} \setminus \mathcal{J}_K}
    \eta_j \Big)^{\alpha} + M^{-\min(2(\gamma-\beta),\gamma)}\Big).
\end{align*}
The proof for the scheme given in~\eqref{HR:MILcomm} can be conducted in 
the same way as for the scheme in~\eqref{HR:Milstein}
except that the projection operator $P_N$ 
in~\eqref{HR:MILcomm} has to be taken into account, see also the 
comments in \cite{MR3842926} and the detailed proof
in~\cite{2015arXiv150908427L}. 
It remains to prove the expression for
the error caused by the approximation 
of the iterated stochastic integrals. We denote $\bar{Y}_m:=\bar{Y}_m^{N,K,M}$ 
for all $m\in\{0,\ldots,M\}$ and compute the following two terms 
\begin{align}\label{HR:Eq:SplitErrorDI}
  \Big( \mathrm{E}\big[\|Y_m-\bar{Y}_m\|_H^2\big]  \Big)^{\frac{1}{2}}
  \leq \Big( \mathrm{E}\big[\|Y_m-Y_{m,\bar{Y}}\|_H^2\big] \Big)^{\frac{1}{2}}
  + \Big( \mathrm{E}\big[\|Y_{m,\bar{Y}}-\bar{Y}_m\|_H^2\big] \Big)^{\frac{1}{2}}
\end{align}
where
\begin{align*}
  Y_{m,\bar{Y}} &=
  P_N\bigg( e^{At_m}X_0+ \sum_{l=0}^{m-1} \int_{t_l}^{t_{l+1}} 
  e^{A(t_m-t_l)}F(\bar{Y}_l)\,\mathrm{d}s
  + \sum_{l=0}^{m-1} \int_{t_l}^{t_{l+1}} e^{A(t_m-t_l)} B(\bar{Y}_l)\,\mathrm{d}W^K_s \\
  &\quad +\sum_{l=0}^{m-1}\int_{t_l}^{t_{l+1}}e^{A(t_m-t_l)}B'(\bar{Y}_l)
  \Big(P_N\int_{t_l}^sB(\bar{Y}_l)\,\mathrm{d}W^K_r\Big)\,\mathrm{d}W_s^K\bigg).
\end{align*}
We insert this expression and obtain
\begin{align*}
  \mathrm{E}\big[\|Y_m-Y_{m,\bar{Y}}\|_H^2\big]  
  &= \mathrm{E}\bigg[\Big\| P_N \Big(
  \sum_{l=0}^{m-1} \int_{t_l}^{t_{l+1}} e^{A(t_m-t_l)}
  \big(F(Y_l)-F(\bar{Y}_l)\big)\,\mathrm{d}s  \\
   &\quad + \sum_{l=0}^{m-1} \int_{t_l}^{t_{l+1}} 
   e^{A(t_m-t_l)}\big( B(Y_l)-B(\bar{Y}_l)\big)\,\mathrm{d}W^K_s  \\
  &\quad
  +\sum_{l=0}^{m-1}
  \Big( \int_{t_l}^{t_{l+1}}e^{A(t_m-t_l)}B'({Y}_l)
  \Big(P_N\int_{t_l}^sB({Y}_l)\,\mathrm{d}W^K_r\Big)\,\mathrm{d}W_s^K  \\
  &\quad
  - \int_{t_l}^{t_{l+1}}e^{A(t_m-t_l)}B'(\bar{Y}_l)
  \Big(P_N\int_{t_l}^sB(\bar{Y}_l)\,\mathrm{d}W^K_r\Big)\,
  \mathrm{d}W_s^K \Big) \Big) \Big\|_H^2 \bigg] \\
  &\leq CMh\sum_{l=0}^{m-1} \int_{t_l}^{t_{l+1}}  \mathrm{E}\bigg[\Big\|
  F(Y_l)-F(\bar{Y}_l)\Big\|_H^2\bigg]\,\mathrm{d}s \\
   &\quad + C\sum_{l=0}^{m-1} \int_{t_l}^{t_{l+1}} 
   \mathrm{E}\bigg[\Big\| B(Y_l)-B(\bar{Y}_l)\Big\|^2_{L_{\text{HS}}(U_0,H)}\bigg]\,\mathrm{d}s  \\
  &\quad
  +C\sum_{l=0}^{m-1}
   \int_{t_l}^{t_{l+1}} \mathrm{E}\bigg[\Big\|B'({Y}_l)
  \Big(P_N\int_{t_l}^sB({Y}_l)\,\mathrm{d}W^K_r\Big) \\
  &\quad
  - B'(\bar{Y}_l)
  \Big(P_N\int_{t_l}^sB(\bar{Y}_l)\,\mathrm{d}W^K_r\Big)\Big\|_{L_{\text{HS}}(U_0,H)}^2 \bigg]\,
  \mathrm{d}s \\
  &\leq
  C_T h \sum_{l=0}^{m-1} \mathrm{E} \Big[\big\|Y_l-\bar{Y}_l\big\|^2_H\Big]
\end{align*}
where the computations are the same as in
\cite[Section 6.3]{MR3320928}.
This estimate mainly employs the Lipschitz
continuity of the involved operators. 
\\ \\
Next, we analyze the second term in~\eqref{HR:Eq:SplitErrorDI}. 
By the stochastic independence of $I^Q_{(i,j),l}$ and $\bar{I}^Q_{(i,j),l}$
from $I^Q_{(i,j),k}$ and $\bar{I}^Q_{(i,j),k}$ for $l \neq k$, we obtain
{\allowdisplaybreaks
\begin{align*}
  \mathrm{E}\big[\|Y_{m,\bar{Y}}-\bar{Y}_m\|_H^2\big] 
  &=
  \mathrm{E}\bigg[\Big\|P_N\Big(\sum_{l=0}^{m-1}
  \sum_{j\in\mathcal{J}_K}e^{A(t_m-t_l)}
  \Big(B'(\bar{Y}_l)\Big(\sum_{i\in\mathcal{J}_K}
  P_N B(\bar{Y}_l)\tilde{e}_iI_{(i,j),l}^Q,\tilde{e}_j \Big)\\
  &\quad
  - B'(\bar{Y}_l)\Big(\sum_{i\in\mathcal{J}_K} P_NB(\bar{Y}_l)
  \tilde{e}_i\bar{I}_{(i,j),l}^Q,\tilde{e}_j\Big) 
 \Big) \Big) \Big\|_H^2 \bigg]\\
 &\leq C \sum_{l=0}^{m-1} \mathrm{E}\bigg[\Big\|
  \int_{t_l}^{t_{l+1}}B'(\bar{Y}_l)
  \Big(\int_{t_l}^s P_NB(\bar{Y}_l)\, 
  \mathrm{d}W_r^K\Big)\, \mathrm{d}W^K_s \nonumber \\
  &\quad
  -\sum_{i,j\in\mathcal{J}_K} \bar{I}_{(i,j),l}^Q B'(\bar{Y}_l) 
  (P_NB(\bar{Y}_l)\tilde{e}_i,\tilde{e}_j)
  \Big\|_H^2\bigg] \nonumber \\
  &= C \sum_{l=0}^{m-1} \mathcal{E}(M,K)^2 \nonumber.
\end{align*}
}%
In total, we get with Gronwall's lemma
\begin{align*}
   \mathrm{E}\big[\|Y_m-\bar{Y}_m\|_H^2\big]
  &\leq C_T h \sum_{l=0}^{m-1} \mathrm{E} \Big[\big\|Y_l-\bar{Y}_l\big\|^2_H\Big]
  + C \sum_{l=0}^{m-1}\mathcal{E}(M,K)^2 \\
  &\leq C M \mathcal{E}(M,K)^2,
\end{align*}
which completes the proof.
\end{proof}
\subsection{Comparison of Computational Cost}
\label{HR:Sec:Cost}
In order to compare the numerical methods introduced in 
this work, we
consider the effective order of convergence based on a 
cost model introduced in~\cite{MR3842926}.
This number combines the theoretical order of convergence, 
as stated for example in Theorem~\ref{HR:Thm:Milstein}, 
with the computational cost involved in the calculation of 
an approximation by a particular scheme. For the computational 
cost model, we assume that each evaluation of a real valued functional 
and each generation of a standard normally distributed random number 
is of some cost $c \geq 1$ whereas each elementary arithmetic 
operation is of unit cost $1$, see~\cite{MR3842926} for 
details. Then, the computational cost for one time step and each 
scheme under consideration can be determined by the corresponding values 
listed in Table~\ref{HR:funcEval-1}.
We compare the two Milstein schemes MIL1 and MIL2 to the 
exponential Euler scheme (EES). 
For the EES, we employ the version 
introduced in~\cite{MR3047942} combined with a 
Galerkin approximation.
The convergence results for the exponential Euler scheme 
in this setting can be obtained similarly as in
the proof of the Milstein scheme in~\cite{MR3320928}, see also 
\cite[Theorem~3.2]{CL16Diss}. We 
state the result without giving a proof.
\begin{prop}[Convergence of EES] \label{HR:Error:EES}
   Assume that (A1)--(A4) hold. Then, there exists a
    constant $C_{T} \in (0,\infty)$, independent of $N$, $K$ and $M$, such
    that for the approximation process $(Y_m^{\text{EES}})_{0 \leq m \leq M}$, 
    defined by the EES, it holds
    \begin{align*}
      & \Big( \mathrm{E}\Big[ \big\| X_{t_m} - Y_m^{\text{EES}} \big\|_H^2
      \Big] \Big)^{\frac{1}{2}} 
      \leq C_{T} \Big( \Big( \inf_{i \in \mathcal{I} \setminus
      \mathcal{I}_N} \lambda_i \Big)^{-\gamma}
      + \Big( \sup_{j \in \mathcal{J} \setminus \mathcal{J}_K}
      \eta_j \Big)^{\alpha} + 
      M^{-q_{\text{EES}}}
      \Big)
    \end{align*}
    with $q_{\text{EES}} = \min(\frac{1}{2},2(\gamma-\beta),\gamma)$
    and for all $m \in \{0,\ldots, M\}$ and all $N,K,M \in
    \mathbb{N}$. 
    The parameters are determined by assumptions (A1)--(A4).
\end{prop}
Note that for the EES we can 
dispense with some of the conditions specified in (A3), e.g., no assumptions
are needed for the second derivative of $B$ and the estimate for
$B'(v) P B(v)-B'(w) P B(w)$ can be suspended.
In the following, let $q$ denote the order of convergence w.r.t.\ 
the step size $h=\frac{T}{M}$. 
Obviously, it holds $q_{\text{MIL}} 
= \min(2(\gamma-\beta),\gamma)\geq\min(\frac{1}{2},2(\gamma-\beta),\gamma)
= q_{\text{EES}}$.
However, we need to take into account the computational cost
in order to determine the scheme that is superior
as we do not need to simulate the iterated integrals 
in the Euler scheme after all. 
Therefore, we derive the effective order of convergence for 
each of the schemes MIL1, MIL2 and EES, see~\cite{MR3842926} for 
details. 
\\ \\
For each approximation $(Y_m)_{m\in\{0,\ldots,M\}}$
under consideration, we minimize the error term 
\begin{equation*}
\sup_{m\in\{0,\ldots,M\}}\big( \mathrm{E}\big[ \| X_{t_m} 
- Y_m \|_H^2  \big] \big)^{\frac{1}{2}} 
\end{equation*}
over all $N,M,K\in\mathbb{N}$ under the constraint  
that the computational cost does not exceed some specified value $\bar{c}>0$.
If we assume that 
$\sup_{j\in\mathcal{J} \setminus \mathcal{J}_K }\eta_j = \mathcal{O}( K^{-\rho_Q})$
and $( \inf_{i\in\mathcal{I}\setminus \mathcal{I}_N}\lambda_i  )^{-1} 
= \mathcal{O}( N^{-\rho_A})$
for some $\rho_A>0$ and $\rho_Q>1$, we obtain the following expression
for all $N,M,K \in\mathbb{N}$ and some $C>0$, see also~\cite{MR3842926},
\begin{equation*}
  \text{err}(\text{SCHEME}) =\sup_{m\in\{0,\ldots,M\}}
  \Big( \mathrm{E}\Big[ \big\| X_{t_m} - Y_m \big\|_H^2  \Big] \Big)^{\frac{1}{2}}
  \leq C \big( N^{-\gamma\rho_A}+K^{-\alpha\rho_Q}+M^{-q} \big).
\end{equation*}
The parameter $q>0$ is determined by the scheme that is considered.
Then, optimization yields the effective
order of convergence, denoted by EOC(SCHEME), which is given as
\begin{equation*}
  \text{err}(\text{SCHEME}) =\mathcal{O}\big(\bar{c}^{\, \, -\text{EOC(SCHEME)}}\big).
\end{equation*}

\begin{table}[b] 
\begin{small}
\begin{center}
  \renewcommand{\arraystretch}{1.4}
  \begin{tabular}{|l|c|c|c|c|}
    \hline
    & \multicolumn{3}{c|}{\# of evaluations of functionals} &
    \\
    Scheme & $\ \ \ P_N F(\cdot)\arrowvert_{H_N} \ \ \ \ $ &
    $\ \ \ P_N B(\cdot)\arrowvert_{U_K} \ \ \ $ &
    $ \ P_N B'(\cdot)\arrowvert_{H_N,U_K} \ $ & \# of $N(0,1)$  r.\ v.\  
    \\
    \hline \hline
    $\operatorname{EES}$ & $N$ & $KN$ & $-$ & $K$ 
    \\ \hline
    $\operatorname{MIL1}$ & $N$ & $KN$ & $KN^2$ & $K(1+2 D)$ 
    \\ \hline
    $\operatorname{MIL2}$ & $N$ & $KN$ & $KN^2$ & $K(1+2 D)+\frac{1}{2} K (K-1)$ 
    \\ \hline
  \end{tabular}
  \renewcommand{\arraystretch}{1.0}
\end{center}
\end{small}
\caption{Computational cost determined by the number of necessary 
evaluations of real-valued functionals and
independent $N(0,1)$-distributed random variables for each 
time step. The choice of $D$ differs for MIL1 and MIL2.}
\label{HR:funcEval-1}
\end{table}
First, we consider Algorithm 1. For the scheme MIL1, the computational 
cost amounts to $\bar{c} = \mathcal{O}(MKN^2)+\mathcal{O}(KM^{2 q_{\text{MIL}}})$, 
see Table~\ref{HR:funcEval-1} and the discussion in the previous section. 
%
We solve the optimization problem and obtain 
\begin{equation} \label{HR:Choice-MNK-MIL1C2}
  \begin{split}
  M &= \mathcal{O}\Big(\bar{c}^{\frac{\gamma\rho_A\alpha\rho_Q}{(2\alpha\rho_Q+\gamma\rho_A)q_{\text{MIL}}+\alpha\rho_Q\gamma\rho_A}}\Big),
  \quad  N = \mathcal{O}\Big(\bar{c}^{\frac{\alpha\rho_Q q_{\text{MIL}}}{(2\alpha\rho_Q+\gamma\rho_A)q_{\text{MIL}}+\alpha\rho_Q\gamma\rho_A}}\Big),\\
   K &= \mathcal{O}\Big(\bar{c}^{\frac{\gamma\rho_A q_{\text{MIL}}}{(2\alpha\rho_Q+\gamma\rho_A)q_{\text{MIL}}+\alpha\rho_Q\gamma\rho_A}}\Big)
   \end{split}
\end{equation}
in the case of 
$\gamma \rho_A (2 q_{\text{MIL}}-1) \leq 2 q_{\text{MIL}}$, denoted as condition M1C2.
These conditions make 
sure that the computational cost is of order $\bar{c}=\mathcal{O}(MKN^2)$. 
Therefore,
we obtain the effective order of convergence from
\begin{equation}\label{HR:StandardEffOrd}
  \text{err}(\operatorname{MIL1}) 
  = \mathcal{O}\Big(\bar{c}^{\, \, -\frac{\gamma\rho_A\alpha\rho_Q q_{\text{MIL}}}{(2\alpha\rho_Q+\gamma\rho_A)q_{\text{MIL}}+\alpha\rho_Q\gamma\rho_A}}\Big),
\end{equation}
which is the same result as for the Milstein scheme
in the case of SPDEs with commutative noise, see the computations in~\cite{MR3842926}.
\\ \\
%
On the other hand, 
in the case of 
$\gamma \rho_A (2 q_{\text{MIL}}-1) \geq 2 q_{\text{MIL}}$, denoted as condition M1C1,
it holds $\bar{c}=\mathcal{O}(KM^{2q_{\text{MIL}}})$ and
optimization yields
\begin{align} \label{HR:Choice-MNK-MIL1C1}
  M = \mathcal{O}\Big(\bar{c}^{\frac{\alpha\rho_Q}
  {(2\alpha\rho_Q+1)q_{\text{MIL}}}}\Big), \quad
  N = \mathcal{O}\Big(\bar{c}^{\frac{\alpha\rho_Q}
  {(2\alpha\rho_Q+1)\gamma\rho_A}}\Big), \quad
  K = \mathcal{O}\Big(\bar{c}^{\frac{1}
  {2\alpha\rho_Q+1}}\Big)
\end{align}
and the effective order of convergence equals
\begin{equation}\label{HR:MIL1EffOrd}
  \text{err}(\operatorname{MIL1}) = \mathcal{O}\Big(\bar{c}^{\, \, -\frac{\alpha\rho_Q}
  {2\alpha\rho_Q+1}}\Big).
\end{equation}

In order to facilitate computation, we 
distinguish the case $(\min_{j\in\mathcal{J}_K} \eta_j)^{-1} = o(K^{\frac{3}{2}})$
which results in $\rho_Q<\frac{3}{2}$ and the case that
$\min_{j\in\mathcal{J}_K} \eta_j = \mathcal{O}(K^{-\frac{3}{2}})$ where 
we choose $\rho_Q \geq \frac{3}{2}$ maximal admissible. In the following, we always 
assume that $\rho_Q$ is chosen maximal such that 
$\sup_{j\in\mathcal{J} \setminus \mathcal{J}_K }\eta_j = \mathcal{O}( K^{-\rho_Q})$
is fulfilled and we refer to 
these two cases by simply writing case $\rho_Q<\frac{3}{2}$ 
and case $\rho_Q \geq \frac{3}{2}$, respectively.
\\ \\
For Algorithm 2, we have to take $\bar{c} = \mathcal{O}(MKN^2)
+ \mathcal{O}(K \min(K^{\frac{3}{2}},K^{\rho_Q})
M^{q_{\text{MIL}}+\frac{1}{2}}) + \mathcal{O}(MK^2)$
into account. 
As above, we need to treat several cases. 
We detail the case $\min(K^{\frac{3}{2}},K^{\rho_Q})=K^{\frac{3}{2}}$, that is, 
$\rho_Q \geq \frac{3}{2}$; the results for $\rho_Q<\frac{3}{2}$ can be obtained 
analogously and are stated in Table~\ref{HR:Tab:Conditions} . 
%
%
The first case corresponds to $\bar{c}=\mathcal{O}(MKN^2)$. 
For $\rho_Q \geq \frac{3}{2}$, $\gamma \rho_A \leq 2 \alpha\rho_Q$ and
$\frac{3}{2} \gamma \rho_A q_{\text{MIL}} + (q_{\text{MIL}}-\frac{1}{2}) \gamma\rho_A\alpha\rho_Q
\leq 2 \alpha\rho_Q q_{\text{MIL}}$, denoted as condition M2C1a, 
we get the same choice for $M$, $N$, $K$ and the same
effective order as for the scheme for SPDEs with 
commutative noise given in~\eqref{HR:Choice-MNK-MIL1C2} and \eqref{HR:StandardEffOrd}.
%
%
In case of $\bar{c}=\mathcal{O}(MK^2)$, that is, 
if $\gamma \rho_A \geq 2 \alpha\rho_Q$ and 
$q_{\text{MIL}} \leq \frac{\alpha \rho_Q}{1+2\alpha\rho_Q }$, denoted as condition M2C2a,
we obtain
\begin{equation} \label{HR:Choice-MNK-MIL2C2a}
  \begin{split}
  M = \mathcal{O}\Big(\bar{c}^{\frac{\alpha\rho_Q}{\alpha \rho_Q+2q_{\text{MIL}}}}\Big),
  \quad  N = \mathcal{O}\Big(\bar{c}^{\frac{\alpha \rho_Q q_{\text{MIL}}}
  {\gamma\rho_A\alpha \rho_Q+2\gamma\rho_A q_{\text{MIL}}}}\Big), \quad
  K = \mathcal{O}\Big(\bar{c}^{\frac{q_{\text{MIL}}}{\alpha \rho_Q+2q_{\text{MIL}}}}\Big)
  \end{split}
\end{equation}
with effective order of convergence given by
\begin{equation} \label{HR:DFM2EffOrd}
  \text{err}(\operatorname{MIL2}) = \mathcal{O}\Big(
  \bar{c}^{\, \, -\frac{\alpha\rho_Q q_{\text{MIL}}}
  {\alpha \rho_Q+2q_{\text{MIL}}}}\Big).
\end{equation}
Note that in this case, it follows $q_{\text{MIL}}<\frac{1}{2}$.
%
Next, we consider the case of 
$2 \alpha \rho_Q q_{\text{MIL}} \leq \frac{3}{2} \gamma \rho_A q_{\text{MIL}}
+ (q_{\text{MIL}}-\frac{1}{2}) \gamma\rho_A \alpha\rho_Q$ and 
$q_{\text{MIL}} \geq \frac{\alpha \rho_Q}{1+2 \alpha \rho_Q}$, 
denoted as condition M2C3a,
i.e., where $\bar{c}=\mathcal{O}(M^{q_{\text{MIL}}+\frac{1}{2}}K^{\frac{5}{2}})$. 
Then, we get
\begin{equation} \label{HR:Choice-MNK-MIL2C3a}
\begin{split}
  M &= \mathcal{O}\Big(\bar{c}^{\frac{\alpha\rho_Q}{\alpha \rho_Q(q_{\text{MIL}}+\frac{1}{2})+\frac{5}{2}q_{\text{MIL}}}}\Big),
  \quad \quad  N = \mathcal{O}\Big(\bar{c}^{\frac{\alpha \rho_Q q_{\text{MIL}}}
  {\gamma\rho_A(\alpha \rho_Q(q_{\text{MIL}}+\frac{1}{2})+\frac{5}{2}q_{\text{MIL}})}}\Big), \\ 
  K &= \mathcal{O}\Big(\bar{c}^{\frac{q_{\text{MIL}}}{\alpha \rho_Q(q_{\text{MIL}}+\frac{1}{2})+\frac{5}{2}q_{\text{MIL}}}}\Big)
\end{split}
\end{equation}
with 
\begin{equation} \label{HR:DFM2EffOrd2}
  \text{err}(\operatorname{MIL2}) = \mathcal{O}\Big(\bar{c}^{\, \, -\frac{\alpha\rho_Q q_{\text{MIL}}}
  {\alpha \rho_Q(q_{\text{MIL}}+\frac{1}{2})+\frac{5}{2}q_{\text{MIL}}}}\Big).
\end{equation}
%
Finally, we want to mention one case for $\rho_Q<\frac{3}{2}$ explicitly where 
we assume $2\alpha q_{\text{MIL}} \leq \gamma \rho_A q_{\text{MIL}} + (q_{\text{MIL}}-\frac{1}{2}) \alpha \gamma \rho_A$
and $q_{\text{MIL}} + \frac{1}{2} \alpha \rho_Q \leq (1+\alpha) \rho_Q q_{\text{MIL}}$, which are the 
conditions denoted as M2C3b. 
In this case, it 
holds that $\bar{c} = \mathcal{O}(M^{q_{\text{MIL}}+\frac{1}{2}} K^{\rho_Q+1})$ which is the only 
case where the dominating term for $\bar{c}$ depends on $\rho_Q$ explicitly. Here we get
\begin{equation} \label{HR:Choice-MNK-MIL2C3b}
  \begin{split}
  M &= \mathcal{O}\Big(\bar{c}^{\frac{\alpha\rho_Q}{\alpha \rho_Q(q_{\text{MIL}}+\frac{1}{2})+q_{\text{MIL}}(\rho_Q+1)}}\Big),
  \quad \quad  N = \mathcal{O}\Big(\bar{c}^{\frac{\alpha \rho_Q q_{\text{MIL}}}
  {\gamma\rho_A(\alpha \rho_Q(q_{\text{MIL}}+\frac{1}{2})+q_{\text{MIL}}(\rho_Q+1))}}\Big), \\ 
  K &= \mathcal{O}\Big(\bar{c}^{\frac{q_{\text{MIL}}}{\alpha \rho_Q(q_{\text{MIL}}+\frac{1}{2})+q_{\text{MIL}}(\rho_Q+1)}}\Big)
  \end{split}
\end{equation}
with 
\begin{equation}\label{HR:DFM2EffOrd3b}
  \text{err}(\operatorname{MIL2}) = \mathcal{O}\Big(\bar{c}^{\, \, -\frac{\alpha\rho_Q q_{\text{MIL}}}
  {\alpha \rho_Q(q_{\text{MIL}}+\frac{1}{2})+q_{\text{MIL}}(\rho_Q+1)}}\Big).
\end{equation}
\begin{table}[tbp]
 \begin{center}
{\small
\renewcommand{\arraystretch}{1.5}
\begin{tabular}{|c|c|c|c|}\hline
  Abbr.\ & Condition & $\bar{c}$ & EOC \\ 
\hline \hline
  M1C1           & $\gamma\rho_A (2q-1) \geq 2q$ & $\mathcal{O}(M^{2q} K)$ & \eqref{HR:MIL1EffOrd} \\ 
\hline
  M1C2           & $\gamma\rho_A (2q -1) \leq 2 q$ & $\mathcal{O}(M K N^2)$ & \eqref{HR:StandardEffOrd} \\ 
\hline
\hline
  M2C1a          & $
  \rho_Q \geq \frac{3}{2} \, \wedge \, \gamma\rho_A \leq 2 \alpha \rho_Q \, \wedge \, 
		\frac{3}{2} \gamma\rho_A q + (q-\frac{1}{2}) \alpha \rho_Q \gamma \rho_A \leq 2 \alpha \rho_Q q 
		$
  & $\mathcal{O}(M K N^2)$ & \eqref{HR:StandardEffOrd} \\ 
\hline
  M2C1b          & $ 
  \rho_Q < \frac{3}{2} \ \wedge \ \gamma\rho_A \leq 2 \alpha\rho_Q \ \wedge \ \gamma \rho_A q + (q-\frac{1}{2}) \alpha \gamma \rho_A \leq 2 \alpha q 
  $
  & $\mathcal{O}(M K N^2)$ & \eqref{HR:StandardEffOrd} \\
\hline
  M2C2a          & $
  \rho_Q \geq \frac{3}{2} \ \wedge \ 2\alpha\rho_Q \leq \gamma\rho_A \ \wedge \ q \leq \frac{\alpha\rho_Q}{2 \alpha\rho_Q +1} 
  $
  & $\mathcal{O}(M K^2)$ & \eqref{HR:DFM2EffOrd} \\ 
\hline
  M2C2b          & $ 
  \rho_Q < \frac{3}{2} \ \wedge \ 2\alpha\rho_Q \leq \gamma\rho_A \ \wedge \ q < \frac{\alpha \rho_Q}{2 \alpha \rho_Q + 2(\rho_Q-1)} 
  $
  & $\mathcal{O}(M K^2)$ & \eqref{HR:DFM2EffOrd} \\ 
\hline
  M2C3a          & $
  \rho_Q \geq \frac{3}{2} \, \wedge \, 2\alpha\rho_Q q \leq \frac{3}{2} \gamma\rho_A q + (q-\frac{1}{2}) \alpha \rho_Q \gamma \rho_A 
  \, \wedge \, q \geq \frac{\alpha \rho_Q}{2 \alpha\rho_Q +1} 
  $ 
  & $\mathcal{O}(M^{q+\frac{1}{2}} K^{\frac{5}{2}})$ & \eqref{HR:DFM2EffOrd2} \\ 
\hline
  M2C3b          & $
  \rho_Q < \frac{3}{2} \ \wedge \ 2\alpha q \leq \gamma \rho_A q + (q-\frac{1}{2}) \alpha \gamma \rho_A 
		\, \wedge \,
		q \geq \frac{\alpha \rho_Q}{2 \alpha \rho_Q + 2(\rho_Q-1)} 
		$ 
  & $\mathcal{O}(M^{q+\frac{1}{2}} K^{\rho_Q+1})$ & \eqref{HR:DFM2EffOrd3b} \\
\hline
\end{tabular}
}
\caption{Conditions M1C1 and M1C2 are the ones that have to be considered for MIL1, 
 whereas the remaining conditions belong to MIL2.
 Under each given condition, the corresponding scheme MIL1 or MIL2 possesses computational cost $\bar{c}$, respectively. 
 Note that $\rho_Q>1$ and we denote $q=q_{\text{MIL}}$.
}
 \label{HR:Tab:Conditions}
 \end{center}
\end{table}
All possible cases M1C1 and M1C2 for MIL1 as well as M2C1a, M2C1b, M2C2a, M2C2b,
M2C3a and M2C3b for MIL2 together with their effective 
orders of convergence are summarized in Table~\ref{HR:Tab:Conditions}. Further, 
the optimal choice for $M$, $N$ and $K$ for the cases not detailed 
is given by the case with the same effective order of convergence listed above.
\\ \\
In order to determine the scheme with the highest effective order of
convergence, we compare the schemes MIL1 and MIL2 to each other 
and to the exponential Euler scheme.
For the EES, the optimal choice for $M$, $N$ and $K$ is given by
\begin{equation} \label{HR:Choice-MNK-EES}
  \begin{split}
  M &= \mathcal{O}\Big(\bar{c}^{\frac{\gamma\rho_A\alpha\rho_Q}{(\alpha\rho_Q+\gamma\rho_A)q_{\text{EES}}+\alpha\rho_Q\gamma\rho_A}}\Big),
  \quad  N = \mathcal{O}\Big(\bar{c}^{\frac{\alpha\rho_Q q_{\text{EES}}}{(\alpha\rho_Q+\gamma\rho_A)q_{\text{EES}}+\alpha\rho_Q\gamma\rho_A}}\Big),\\
   K &= \mathcal{O}\Big(\bar{c}^{\frac{\gamma\rho_A q_{\text{EES}}}{(\alpha\rho_Q+\gamma\rho_A)q_{\text{EES}}+\alpha\rho_Q\gamma\rho_A}}\Big)
   \end{split}
\end{equation}
with the effective order of convergence
\begin{equation}\label{HR:EESEffOrd}
  \text{err}(\operatorname{EES}) = \mathcal{O}\Big(
  \bar{c}^{\, \, -\frac{q_{\text{EES}}\gamma\rho_A\alpha\rho_Q}{(\alpha\rho_Q+\gamma\rho_A)q_{\text{EES}}
  +\gamma\rho_A\alpha\rho_Q}}\Big)
\end{equation}
where $q_{\text{EES}} = \min(\frac{1}{2},2(\gamma-\beta),\gamma)$, see~\cite{MR3842926}.
\\ \\
Obviously, our main interest is in parameter 
constellations such that $q_{\text{MIL}} > q_{\text{EES}}$ which implies 
that $q_{\text{EES}} = \frac{1}{2}$. In case of $q_{\text{MIL}} = q_{\text{EES}} \leq \frac{1}{2}$
the EES is always the optimal choice compared to MIL1 and MIL2.
Therefore, we assume $q_{\text{MIL}} > q_{\text{EES}}=\frac{1}{2}$ in the following. 
Then, by comparing the different effective orders of convergence 
across parameter sets, one can show that except for one case
the Milstein scheme always has a higher effective order of convergence
than the exponential Euler scheme.
We refer to Table~\ref{HR:Tab:CompareEffOrder} for an overview; this 
shows that for larger $q_{\text{MIL}}$ the Milstein scheme is 
favoured over the exponential Euler scheme. 
Here, we only elaborate one case. Assume that
the parameters take values such that either the scheme 
MIL1 or the scheme MIL2 obtains the same effective order of convergence as the 
scheme for SPDEs with commutative noise~\eqref{HR:StandardEffOrd}. Note that 
\eqref{HR:StandardEffOrd} is the highest effective order that can be attained by 
MIL1 and MIL2 for $q_{\text{MIL}}>\frac{1}{2}$ anyway.
We compare the effective order~\eqref{HR:StandardEffOrd} 
with that of the exponential Euler scheme in \eqref{HR:EESEffOrd}
\begin{equation*}
  \frac{q_{\text{MIL}} \gamma \rho_A \alpha \rho_Q}{(2\alpha\rho_Q+\gamma\rho_A) q_{\text{MIL}}
  + \gamma \rho_A \alpha\rho_Q} \substack{<\\>}
  \frac{q_{\text{EES}}\gamma\rho_A\alpha\rho_Q}{(\alpha\rho_Q+\gamma\rho_A)q_{\text{EES}}
  +\gamma\rho_A\alpha\rho_Q}.
\end{equation*}
This can be rewritten such that we obtain
\begin{equation*}
 q_{\text{MIL}} (\gamma\rho_A-q_{\text{EES}}) \substack{<\\ >}q_{\text{EES}}\gamma\rho_A.
\end{equation*}
For $q_{\text{MIL}}>q_{\text{EES}} = \frac{1}{2}$, this results in
\begin{equation*}
 \gamma \rho_A \substack{< \\ >} \frac{q_{\text{MIL}}}{2q_{\text{MIL}}-1}.
\end{equation*}
The condition $\gamma \rho_A > \frac{q_{\text{MIL}}}{2q_{\text{MIL}}-1}$ is required for a higher 
effective order of the Milstein scheme whereas 
$\gamma \rho_A \leq \frac{q_{\text{MIL}}}{2q_{\text{MIL}}-1}$ results in a higher 
order for the exponential Euler scheme.
Clearly, either condition M1C1 or condition M1C2 has to be fulfilled and
in case of M1C1 the effective order of convergence for MIL1 in \eqref{HR:MIL1EffOrd} 
is greater than that in \eqref{HR:EESEffOrd} for the EES scheme if $q_{\text{MIL}} > q_{\text{EES}}$.
Thus, in the case that M1C1 is fulfilled it only remains to check whether MIL2 
attains an even higher effective order of convergence than \eqref{HR:MIL1EffOrd}. 
These calculations can be conducted in a similar way as above.
\\ \\
\begin{table}[tbp]
\begin{center}
{\small
\renewcommand{\arraystretch}{1.5}
 \begin{tabular}{|c|c|c|c|}\hline 
  Conditions & Optimal scheme & Optimal $M$, $N$, $K$ & \ EOC \ \\ 
  \hline\hline
  M1C1 $\wedge$ M2C1a & MIL2 & \eqref{HR:Choice-MNK-MIL1C2} & \eqref{HR:StandardEffOrd} \\ 
  \hline 
  M1C1 $\wedge$ M2C1b & MIL2 & \eqref{HR:Choice-MNK-MIL1C2} & \eqref{HR:StandardEffOrd} \\ 
  \hline 
  M1C1 $\wedge$ M2C3a $\wedge$ $(2\alpha \rho_Q-3) q < \alpha \rho_Q$ & MIL1 & \eqref{HR:Choice-MNK-MIL1C1} & \eqref{HR:MIL1EffOrd} \\ 
  \hline 
  M1C1 $\wedge$ M2C3a $\wedge$ $(2\alpha \rho_Q-3) q \geq \alpha \rho_Q$ & MIL2 & \eqref{HR:Choice-MNK-MIL2C3a} & \eqref{HR:DFM2EffOrd2} \\ 
  \hline 
  M1C1 $\wedge$ M2C3b $\wedge$ $\alpha (2q-1) < 2q$ & MIL1 & \eqref{HR:Choice-MNK-MIL1C1} & \eqref{HR:MIL1EffOrd} \\
  \hline 
  M1C1 $\wedge$ M2C3b $\wedge$ $\alpha (2q-1) \geq 2q$ & MIL2 & \eqref{HR:Choice-MNK-MIL2C3b} & \eqref{HR:DFM2EffOrd3b} \\
  \hline 
  M1C2 $\wedge$ $\gamma \rho_A (2q-1) \leq q$ &  EES & \eqref{HR:Choice-MNK-EES} & \eqref{HR:EESEffOrd} \\
  \hline 
  M1C2 $\wedge$ M2C1a $\wedge$ $\gamma \rho_A (2q-1) > q$ & MIL1=MIL2 & \eqref{HR:Choice-MNK-MIL1C2} & \eqref{HR:StandardEffOrd} \\
  \hline 
  M1C2 $\wedge$ M2C1b $\wedge$ $\gamma \rho_A (2q-1) > q$ & MIL1=MIL2 & \eqref{HR:Choice-MNK-MIL1C2} & \eqref{HR:StandardEffOrd} \\
  \hline 
  M1C2 $\wedge$ M2C3a $\wedge$ $\gamma \rho_A (2q-1) > q$ & MIL1 & \eqref{HR:Choice-MNK-MIL1C2} & \eqref{HR:StandardEffOrd} \\
  \hline 
  M1C2 $\wedge$ M2C3b $\wedge$ $\gamma \rho_A (2q-1) > q$ & MIL1 & \eqref{HR:Choice-MNK-MIL1C2} & \eqref{HR:StandardEffOrd} \\
  \hline 
 \end{tabular}
}
\caption{For a given parameter set, the conditions in this table
have to be checked in order to determine the optimal scheme 
among the schemes EES, MIL1 and MIL2 for the case of $q=q_{\text{MIL}} > q_{\text{EES}} = \frac{1}{2}$. 
In case of $q_{\text{MIL}} =q_{\text{EES}} \leq \frac{1}{2}$, the exponential Euler scheme is always the optimal choice.
}
\label{HR:Tab:CompareEffOrder}
\end{center}
\end{table}%
Based on the effective order of convergence, it is not
possible to identify one scheme 
that dominates
the others across all parameter constellations. 
The results of a comparison are summarized in 
Table~\ref{HR:Tab:CompareEffOrder}; this overview
clearly illustrates the
dependence on the parameters $q_{\text{MIL}}$, $\alpha$, $\gamma$, $\rho_A$ and $\rho_Q$.
For completeness, we want to note that parts of (A3) do not have to be 
fulfilled for the exponential Euler scheme. 
Therefore, there exist equations where this scheme might indeed
be beneficial for parameter sets other than 
the combinations stated in Table~\ref{HR:Tab:CompareEffOrder}.
The effective order for the Milstein scheme indicates that, compared to the Euler schemes,
the increase in the computational cost that results from the approximation
of the iterated stochastic integrals is, in most cases,
significantly compensated by the higher theoretical order of convergence $q_{\text{MIL}}$ w.r.t.\ 
the time steps that the Milstein scheme attains.
\subsection{Example}
Finally, we illustrate the theoretical 
results on the effective order of convergence and the consequences 
for the choice of a particular scheme, 
summarized in Table~\ref{HR:Tab:CompareEffOrder}, with 
an example.
\\ \\
Throughout this section, we fix the following setting.
Let $H=U=L^2((0,1),\mathbb{R})$, set $T=1$, $\beta=0$
and $\mathcal{I}=\mathcal{J} =\mathbb{N}$.
We choose $A$ to be the Laplacian with Dirichlet 
boundary conditions; to be precise, $A=\frac{1}{100} \Delta$.
Thus, it holds for the eigenvalues $\lambda_i = \frac{\pi^2i^2}{100}$,
for the eigenvectors $e_i= \sqrt{2} \sin(i\pi x)$ for 
$i\in\mathbb{N}$, $x\in(0,1)$
and on the boundary, we have
$X_t(0) = X_t(1) = 0$ for all $t\in(0,T]$.
The operator $Q$ is defined by $\eta_j = j^{-3}$ 
and $\tilde{e}_j = \sqrt{2} \sin(j\pi x)$ for $j\in\mathbb{N}$, $x\in(0,1)$. 
As a result of this, it holds $\rho_A=2$ and $\rho_Q=3$.
Moreover,
we choose $F(y) = 1-y$, $y\in H$ and $\xi(x)= X_0(x) =0$ 
for all $x\in(0,1)$. The operator 
$B$ is defined in the following.
It fits into
the general setting introduced for
the numerical analysis in~\cite[Sec.~5.3]{MR3842926},
which we repeat here in short only. 
Let some functionals $\mu_{ij}:H_{\beta}\rightarrow \mathbb{R}$, 
$\phi_{ij}^k:H_{\beta}\rightarrow \mathbb{R}$ 
be given for $i,k\in \mathcal{I}$, $j\in\mathcal{J}$
such that $\phi_{ij}^k$ is the Fr\'{e}chet derivative of $\mu_{ij}$
in direction $e_k$. Then, we define
\begin{align*}
  B(y)u =\sum_{i\in\mathcal{I}}\sum_{j\in\mathcal{J}} 
  \mu_{ij}(y)\langle u,\tilde{e}_j\rangle_U e_i
\end{align*}
and it holds that
\begin{align*}
  \big( B'(y)(B(y)v) \big)u  =\sum_{i,k\in\mathcal{I}}\sum_{j,r\in\mathcal{J}} 
  \phi_{ij}^k(y)\mu_{kr}(y)\langle v,\tilde{e}_r\rangle_U \langle u,\tilde{e}_j\rangle_U e_i
\end{align*}
for $y\in H_{\beta}$ and $u,v\in U_0$. 
For details, we refer to~\cite[Sec.~5.3]{MR3842926}.
\\ \\
Here, we choose $\mu_{ij}(y) = \frac{\langle y,e_j\rangle_H}{i^4+j^4}$ 
for all $i\in\mathcal{I}$,
$j\in\mathcal{J}$ and $y\in H$.
With this choice, we get
$\phi_{ij}^k(y) = \left\{\begin{array}{rr} 0, & k \neq j \\ \frac{1}{i^4+j^4},
& k = j \end{array}\right.$ for all $i,k \in\mathcal{I}$, $j\in \mathcal{J}$, $y\in H$.
We show that assumptions (A1)--(A4) are fulfilled in this setting.
For conditions (A1), (A2) and 
(A4) this is obvious. It remains to examine (A3). 
We use the expressions that have been computed
in~\cite[Sec.~5.3]{MR3842926}, that is,
\begin{equation*}
  \|B(y)\|_{L(U,H_{\delta})}\leq  \sum_{i\in\mathcal{I}} \sum_{j\in\mathcal{J}} \lambda_i^{\delta}
  \vert\mu_{ij}(y) \vert \leq \sum_{i\in\mathcal{I}} \sum_{j\in\mathcal{J}}
  \frac{1}{j^{2-2\delta}}\frac{\pi^{2\delta}}{100^{\delta}}\frac{1}{i^{2-2\delta}}  \|y\|_{H_{\delta}}
\end{equation*}
for all $y\in H_{\delta}$.
Thus,  we get $\|B(y)\|_{L(U,H_{\delta})}\leq C(1+\|y\|_{H_{\delta}})$ 
for all $y\in H_{\delta}$
if $\delta <\frac{1}{2}$, where we select the maximal value for $\delta$.
Moreover, we check
\begin{align*}
  \|(-A)^{-\vartheta}B(z)Q^{-\alpha}\|_{L_{\text{HS}}(U_0,H)} &= \Big(\sum_{j\in\mathcal{J}}
  \eta_j^{1-2\alpha} \sum_{i\in\mathcal{I}} \lambda_i^{-2\vartheta} \mu_{ij}^2(z)\Big)^{\frac{1}{2}}\\
  &\leq C\Big(\sum_{j\in\mathcal{J}} \frac{1}{j^{3(1-2\alpha)+8+4\gamma}} 
  \sum_{i\in\mathcal{I}}
  \frac{1}{i^{4\vartheta}} \|z\|_{H_{\gamma}}^2\Big)^{\frac{1}{2}}
\end{align*}
for all $z\in H_{\gamma}$. This shows that 
$\|(-A)^{-\vartheta}B(z)Q^{-\alpha}\|_{L_{\text{HS}}(U_0,H)}
\leq C(1+\|z\|_{H_{\gamma}})$ is fulfilled 
for all $z\in H_{\gamma}$ if $\alpha <\frac{7}{3}$.
The remaining conditions in (A3) hold as well.
These are not stated here as they do not restrict the parameters. 
Finally, we show that the commutativity condition~\eqref{HR:Comm}, 
expressed in the notation presented above,
is actually not fulfilled.
On the one hand, we get
\begin{align*}
  \sum_{k\in\mathcal{I}} \phi_{im}^k(y)\mu_{kn}(y)=  \frac{1}{i^4+m^4}
  \frac{\langle y,e_n\rangle_H}{m^4+n^4}
\end{align*}
but
\begin{align*}
  \sum_{k\in\mathcal{I}} \phi_{in}^k(y)\mu_{km}(y) =  \frac{1}{i^4+n^4}
  \frac{\langle y,e_m\rangle_H}{n^4+m^4}
\end{align*}
holds for $y\in H$ and $i\in\mathcal{I}$, $n,m\in\mathcal{J}$. Obviously, 
these two terms are not equal for all $n,m\in\mathcal{J}$.
\\ \\
From the parameter values stated above, we compute
$\gamma \in [\frac{1}{2},1)$. With this information, we can identify the scheme that is superior 
according to Table~\ref{HR:Tab:CompareEffOrder}. Let $\varepsilon\in(0,\frac{1}{2})$
be arbitrarily small and choose
$q_{\text{MIL}}=\gamma=1-\varepsilon>q_{\text{EES}}$ and $\alpha =\frac{7}{3}-\varepsilon$. First, we check 
condition M1C2, see Table~\ref{HR:Tab:Conditions}, which holds as
\begin{equation*}
 \gamma\rho_A (2q_{\text{MIL}}-1) \leq 2q_{\text{MIL}}
  \quad 
  \Leftrightarrow
  \quad
  2(1-\varepsilon)(1-2\varepsilon) \leq 2(1-\varepsilon) .
\end{equation*}
Moreover, condition M2C1a in Table~\ref{HR:Tab:Conditions} is fulfilled as well because 
it holds $\rho_Q=3$, $\gamma \rho_A 
= 2(1-\varepsilon) \leq 6(\frac{7}{3}-\varepsilon) =  2 \alpha \rho_Q$
and it is easy to check that
\begin{align*}
  \frac{3}{2} \gamma \rho_A q_{\text{MIL}} + \Big(q_{\text{MIL}}-\frac{1}{2} \Big) \gamma \rho_A \alpha \rho_Q &\leq 2 \alpha \rho_Q q_{\text{MIL}} \\
  \Leftrightarrow \quad 
  3 (1-\varepsilon) + 6 \Big(\frac{1}{2}-\varepsilon \Big) \Big(\frac{7}{3}-\varepsilon \Big) &\leq 6 \Big(\frac{7}{3}-\varepsilon \Big)
\end{align*}
is fulfilled due to $\gamma=q_{\text{MIL}}$, which proves condition M2C1a.
From Table~\ref{HR:Tab:CompareEffOrder}, we expect that both schemes MIL1 and MIL2 
obtain the same effective order of convergence \eqref{HR:StandardEffOrd} which exceeds the order of the 
exponential Euler scheme in this case.
For some fixed $N\in\mathbb{N}$, we compute the relation of $N,M,K$ from 
\eqref{HR:Choice-MNK-MIL1C2}. This yields $M = N^2$ and 
$K= \lceil N^{\frac{2}{7}} \rceil$ for the Milstein schemes.
Moreover, we calculate the effective order of convergence as
$\text{error}(\operatorname{MIL1}) = \text{error}(\operatorname{MIL2})= \mathcal{O}(\bar{c}^{-\frac{7}{15}
+\varepsilon})$
for some arbitrarily small $\varepsilon >0$.
For the EES, on the other hand, we obtain $M=N^4$, $K=\lceil N^{\frac{2}{7}} \rceil$ and 
$\text{error}(\operatorname{EES}) = \mathcal{O}(\bar{c}^{-\frac{14}{37}+\varepsilon})$.
\begin{figure}[H]
\begin{center}
 \includegraphics[scale=0.6]{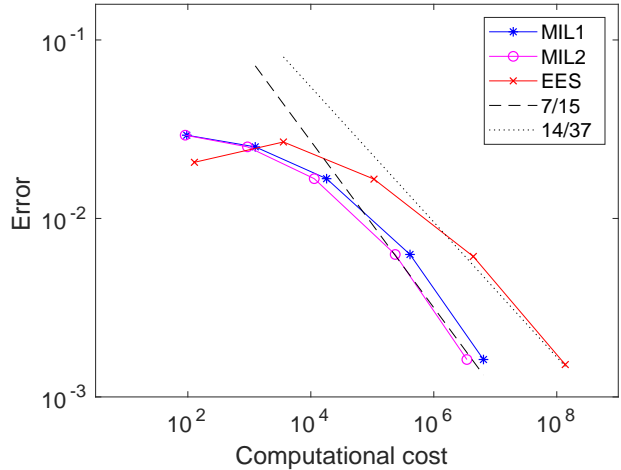}
\caption[Error against computational cost]{Error against computational cost computed from 200 
paths for $N \in\{ 2, 4, 8, 16, 32\}$ in log-log scale.}
\label{HR:Plot:NC1}
\end{center}
\end{figure}%
\begin{table}[t]
\begin{center}
\renewcommand{\arraystretch}{1.2}
\begin{tabular}{|p{0.65cm}|p{0.65cm}|p{0.65cm}||p{1.4cm}|p{1.4cm}|p{1.4cm}||p{1.4cm}|p{1.4cm}|p{1.4cm}|}\hline
\multicolumn{3}{|c||}{} & \multicolumn{3}{c||}{MIL1}      &    \multicolumn{3}{c|}{MIL2}       \\ \hline
 $N$   & $M$       &   $K$                & $\bar{c}$                        &Error                  & Std                   &  $\bar{c}$                          & Error                  & Std     \\ \hline
 2   & 4       & $\lceil 2^{\frac{2}{7}} \rceil$   & 64             		& $2.9\cdot 10^{-2}$	&  $7.2\cdot 10^{-3}$   & 71			       &$2.9\cdot 10^{-2}$	&  $7.2\cdot 10^{-3}$   \\ \hline
 4   & $2^4$   & $\lceil 2^{\frac{4}{7}} \rceil$   &1024			& $2.5\cdot 10^{-2}$ 	&  $5.4\cdot 10^{-4}$	& 758		               &$2.5\cdot 10^{-2}$ 	&  $5.4\cdot 10^{-4}$  \\ \hline		
 8   & $2^6$   & $\lceil 2^{\frac{6}{7}} \rceil$     &16384			& $1.7\cdot 10^{-2}$	&  $1.1\cdot 10^{-4}$   & 9897 		       &$1.7\cdot 10^{-2}$	&  $1.1\cdot 10^{-4}$   \\ \hline
 16  & $2^8$   & $\lceil 2^{\frac{8}{7}} \rceil$   &393216			& $6.3\cdot 10^{-3}$ 	&  $2.8\cdot 10^{-5}$ 	& 220196	               &$6.3\cdot 10^{-3}$ 	&  $2.8\cdot 10^{-5}$   \\ \hline
 32  & $2^{10}$   & $\lceil 2^{\frac{10}{7}} \rceil$   &6291456			& $1.6\cdot 10^{-3}$ 	&  $2.6\cdot 10^{-5}$ 	& 3325212	               &$1.6\cdot 10^{-3}$ 	&  $2.6\cdot 10^{-5}$   \\ \hline
\end{tabular}
\quad \\[0.2cm]
\renewcommand{\arraystretch}{1.2}
\begin{tabular}{|p{0.65cm}|p{0.65cm}|p{0.65cm}||p{1.8cm}|p{1.4cm}|p{1.4cm}|}\hline
\multicolumn{3}{|c||}{}  & \multicolumn{3}{|c|}{Exponential Euler}   \\ \hline
$N$    &$M$       &   $K$         & $\bar{c}$                   & Error                 & Std   \\ \hline
2    &$2^4$   &$\lceil 2^{\frac{2}{7}} \rceil$  & 64		&$2.1\cdot 10^{-2}$ 	&  $6.0\cdot 10^{-3}$  \\ \hline
4    &$2^8$   &$\lceil 2^{\frac{4}{7}} \rceil$   & 2048		&$2.7\cdot 10^{-2}$ 	   &  $7.2\cdot 10^{-4}$  \\ \hline		
8    &$2^{12}$&$\lceil 2^{\frac{6}{7}} \rceil$   & 65536	&$1.7\cdot 10^{-2}$ 	   &  $2.1\cdot 10^{-4}$  \\ \hline
16   &$2^{16}$&$\lceil 2^{\frac{8}{7}} \rceil$   & 3145728	&$6.1\cdot 10^{-3}$        &  $4.4\cdot 10^{-5}$  \\ \hline
32   &$2^{20}$&$\lceil 2^{\frac{10}{7}} \rceil$   & 100663296	&$1.5\cdot 10^{-3}$        &  $6.6\cdot 10^{-6}$  \\ \hline
\end{tabular}
 \end{center}
\caption[Error and standard deviation]{Error and standard deviation obtained from 200 paths. 
The computational cost $\bar{c}$ is computed as
$\bar{c}(\operatorname{MIL1}) = MKN^2+KM^{2q_{\text{MIL}}}+M(K+N+KN)$, 
$\bar{c}(\operatorname{MIL2}) = MKN^2+M^{q_{\text{MIL}}+\frac{1}{2}}K^{\frac{5}{2}}+MK^2+M(K+N+KN)$
and $\bar{c}(\operatorname{EES}) =  MKN+MN+MK$.}
\label{HR:Tab:NC1}
\end{table}%

In the numerical analysis, we simulate 200 paths with the schemes MIL1, MIL2 
and EES.
The results are compared to a substitute for the exact solution --
an approximation computed with the linear implicit Euler scheme \cite{MR1825100}
with $N = 2^5$, $K= \lceil 2^{\frac{10}{7}} \rceil$ and $M = 2^{16}$.
Our findings are summarized in Table~\ref{HR:Tab:NC1} and Figure~\ref{HR:Plot:NC1}.
In Figure~\ref{HR:Plot:NC1}, we plot the errors versus the computational cost based on
the cost model that is used for the analysis. Here, one observes that 
the Milstein schemes obtain a higher effective order of convergence than the Euler scheme.
Moreover, Table~\ref{HR:Tab:NC1} illustrates the difference in the computational costs of these schemes.
The Euler scheme involves costs which are significantly higher. A comparison of MIL1 and MIL2 shows 
for this example that the Milstein scheme in combination with Algorithm~2 involves a lower computational 
cost than the Milstein scheme combined with Algorithm~1.
%
%

%
\bibliographystyle{spmpsci}
%

%
%
%
\end{document}